\documentclass[]{amsart}

\usepackage{amssymb, mathrsfs, color, amsthm, tikz, subcaption}

\newtheorem{thm}{Theorem}[section]
\newtheorem{cor}[thm]{Corollary}
\newtheorem{lem}[thm]{Lemma}

\newtheorem{prop}[thm]{Proposition}

\newtheorem{rem}[thm]{Remark}

\numberwithin{equation}{section}

\renewcommand{\phi}{\varphi}

\newcommand{\D}{\mathbb{D}}

\newcommand{\Dir}{\mathcal{D}}

\newcommand{\C}{\mathbb{C}}

\newcommand{\eps}{\varepsilon}

\newcommand{\B}{\mathcal{B}}

\newcommand{\diam}{\mbox{\rm diam\,}}

\begin{document}

\title[Composition Semigroups on the Besov Spaces]
{Composition Semigroups on the Besov Spaces}
\date{\today}

\author[A. Anderson]{Austin Anderson}
\address{Department of Mathematics, University of Hawaii, Honolulu, Hawaii 96822}
\email{austina@hawaii.edu}
\author[M. Jovovic]{Mirjana Jovovic}
\address{Department of Mathematics, University of Hawaii, Honolulu, Hawaii 96822}
\email{jovovic@math.hawaii.edu}
\author[W. Smith]{Wayne Smith}
\address{Department of Mathematics, University of Hawaii, Honolulu, Hawaii 96822}
\email{wayne@math.hawaii.edu}

\thanks {}

\begin{abstract}
 We study semigroups of
 composition operators acting on the Besov spaces $\B_p$, where they exhibit some new behaviors relative to many classical spaces.  Often for a Banach space $X$ of analytic functions on the unit disk, the maximal closed space of strong continuity, $[ \phi_t, X ]$, exists for every semigroup $\{ \phi_t \}$ of analytic self-maps of the disk, and the question whether $[\phi_t , X ]$ equals $X$ itself has an answer independent of $\{\phi_t\}$.  Such is the case for the Hardy and Bergman spaces, Bloch, BMOA, and $H^{\infty}$.  For the disk algebra $A$, $[\phi_t , A ] = A$ precisely when $\{\phi_t\} \subset A$.  For $\B_p$ with $p \geq 2$, every $\{\phi_t\} \subset \B_p$ and always $[ \phi_t, \B_p ] = \B_p$, but this fails when $1 < p < 2$.  We give an example where $\{\phi_t\} \subset \B_p$ and yet the induced composition operators $\{C_t\}$ are not bounded on $\B_p$ and
 we do not know if $[\phi_t,\B_p]$ exists.  If it does exist,  it cannot be equal to $\B_p$. Under the hypothesis that there is a uniform bound for the operator norms of the $\{C_t\}$, $0 \leq t \leq 1$,
 we characterize the semigroups $\{ \phi_t \}$ such that $[ \phi_t, \B_p ] = \B_p$.
\end{abstract}
\maketitle

\section{Introduction}

We denote by $\D$ the unit disk $\{z : |z| < 1\}$, and by $H(\D)$ the set of analytic functions on $\D$.
 A family $\{\phi_t\}_{t \geq 0}$ of
analytic functions $\phi_t : \D \to \D$ is called a one-parameter semigroup if it satisfies the following three conditions:
\begin{enumerate}
\item [(SG1)] $\phi_0(z)=z$, $z\in \D$;
\item [(SG2)] $\phi_{s+t}=\phi_s\circ\phi_t$, for all $t,s\ge 0$;
\item [(SG3)] the mapping $(t,z)\to\phi_t(z)$ is continuous on $[0,\infty) \times \D$.
\end{enumerate}

We say that $\{\phi_t\}$ is nontrivial unless $\phi_t(z)=z$ for all $t\ge0$.  For any analytic $\phi : \D \to \D$ we define the linear operator $C_\phi : H(\D) \to H(\D)$ by $C_{\phi}(f) = f \circ \phi$.
Given a semigroup $\{\phi_t\}$, we associate a set of linear operators $\{ C_t \}$, where $C_t = C_{\phi_t}$.  Then
$\{ C_t \}$ is called a composition semigroup, as $C_0$ is the identity operator and $C_{s+t}=C_s C_t$.
If, for all $t \geq 0$, $C_t$ is a bounded operator on some Banach space $X\subset H(\D)$, we say that
the semigroup $\{\phi_t\}$ acts on $X$.  If additionally
the strong continuity condition
  \begin{align} \label{SCC}
  \lim_{t \to 0^+} \| f\circ \phi_t - f \|_X = 0
  \end{align}
holds for all $f \in X$, then it is said that $\{\phi_t\}$ \textit{generates} $\{ C_t \}$, and $\{C_t\}$ is strongly continuous on $X$.

Composition semigroups were introduced in \cite{BP}, where it is shown that the $\phi_t$ are univalent and every semigroup generates a composition semigroup on the Hardy spaces $H^p$, $1 \leq p < \infty$.

\subsection{The maximal subspace of strong continuity}
When the semigroup $\{\phi_t\}$ does not generate $\{ C_t \}$ on $X$, it is natural to consider
$$
(\phi_t,X)=\{f\in X\, :\, \lim_{t \to 0^+} \| f\circ \phi_t - f \|_X = 0\},
$$
which is clearly the maximal subspace of  $X$ on which the strong continuity condition (\ref{SCC}) holds.
However, as we will see below, $(\phi_t,X)$ may not be closed and hence may not be a Banach space.
Also, the subspace $(\phi_t,X)$ may not be invariant under
elements of the composition semigroup $\{ C_t \}$.
Hence $\{\phi_t\}$ may not generate a semigroup of operators on $(\phi_t,X)$.

Thus it is also natural to consider the maximal  closed subspace of strong continuity
 on which $\{\phi_t\}$ acts, when it exists.
This was introduced in \cite{Bla} for the case $X$=BMOA, with
 the existence  proved in \cite[Prop 2.1]{Bla}, and denoted by $[\phi_t,\text{BMOA}]$.
Key to the proof is
 that $\sup_{0\le t\le 1}\| C_t \|_{\text{BMOA}}<\infty$ for every composition
 semigroup,  which is used to show that $(\phi_t,\text{BMOA})$ is closed.
That $(\phi_t,\text{BMOA})$ is invariant under each $C_t$ is a consequence of
 every composition operator being bounded on BMOA.

 Here, we formulate a version of the result for a general
 Banach space $X\subset H(\D)$:

\begin{prop} \label{subspSCC}
Suppose the  semigroup $\{\phi_t\}$ acts on a Banach space $X \subset H(\D)$ and
$\sup_{0\le t\le 1}\| C_t \|_{X} <\infty$. Then $(\phi_t,X)$
is a closed subspace of $X$ and $\{\phi_t\}$ generates a semigroup of
operators on $(\phi_t,X)$.  Moreover, $(\phi_t,X)$ is the maximal such subspace of $X$.
\end{prop}

 The proof is the same as that of \cite[Prop 2.1]{Bla}, and so will be omitted.
Following \cite{Bla}, when this maximal closed subspace exists,
 we denote it by $[\phi_t,X]$.

  \medskip

For some  $X$, $[\phi_t,X]=X$ for all composition semigroups $\{\phi_t\}$.
This is the case for the Hardy spaces $H^p$, $1\le p<\infty$, the weighted
Bergman spaces $A^p_\alpha$, $1\le p<\infty$, $\alpha>-1$, and the Dirichlet space $\Dir$;
see \cite{Sis}.

For some  $X$, $[\phi_t,X]\subsetneqq X$ for all nontrivial composition semigroups $\{\phi_t\}$.
This is the case for  $H^\infty$, BMOA,  and the Bloch space;
see \cite{AJS2}.

For the disk algebra $A$, it is well known that $\{\phi_t\}$ acts on $A$ if and only if
$\{\phi_t\}\subset A$, and moreover it is known that this is also equivalent to $[\phi_t, A]=A$; see \cite{CD}.

In this paper our main purpose is to study the maximal spaces of strong continuity for semigroups
on the classical Besov spaces $\B_p$, $1<p$.
$\B_2$ is the Dirichlet space $\Dir$,  and
$\Dir\subset \B_p$, $2\le p$.  It follows that, for  $2\le p$, every semigroup $\{\phi_t\}\subset\B_p$,
 and we will show that  $[\phi_t, \B_p]=\B_p$.

When $1< p < 2$, the Besov spaces are too small to contain all univalent self maps of $\D$,
and the situation is more complicated.
We will present several examples that show what can happen.
First, we give an example of a semigroup
$\{\phi_t\}$ that does not act on any $\B_p$, $1<p<2$; see
Proposition \ref{compgpProp}.
Our next main result is that for each $p$, $1<p<2$, there is a semigroup $\{\phi_t\}\subset \B_p$ such that
 $(\phi_t,\B_p)$ is dense in $\B_p$,  but $\{\phi_t\}$ does not act on $\B_p$;
see Corollary \ref{PnotB_p}.  We end the paper with an
example of a nontrivial semigroup $\{\psi_t\}$ such that $[\psi_t, \B_p]=\B_p$, $1<p<2$;
see Theorem \ref{[]=B_p}.

\vskip .2in

{\it Notation for constants.}
For $X$ and $Y$ nonnegative quantities, the notations $X\lesssim Y$  and $Y\gtrsim X$
mean that $X\le CY$, where the exact value of the constant $C$ is not important. Similarly,  $X\approx Y$ means that
both $X\lesssim Y$ and $Y\lesssim X$ hold.

\section{Background}

\subsection{Some geometric function theory}

For a simply connected domain $G \subsetneqq \C$ with Riemann map $g: \D \to G$, let $\rho_{G}$ denote the hyperbolic metric on $G.$ For $w_1, w_2 \in G$,
  $$\rho_{G}(w_1,w_2) = \inf_{\gamma}  \int_{\gamma} \frac{2}{1-|z|^2}|dz|,$$
where $\gamma \subset \D$ ranges over curves from $g^{-1}(w_1)$ to $g^{-1}(w_2)$.
This infimum is achieved by a curve $\gamma$, and $\alpha = g(\gamma)$  is called the hyperbolic geodesic or noneuclidean segment from $w_1$ to $w_2$.

We denote the hyperbolic disk centered at $w_0$ with radius $r$ by
$$
\Delta_G(w_0,r) = \{w \in G : \rho_G(w,w_0) < r\}.
$$
and the euclidean disk by
$$
B(z_0,r) = \{z \in \C : |z-z_0| < r\}.
$$
Since the hyperbolic metric is conformally invariant, so are
hyperbolic disks and hyperbolic geodesics.

For a domain $\Omega \subsetneqq \C$ and $w_0 \in \Omega$, we denote the distance to the boundary by
$$
\delta_{\Omega}(w_0) = \inf_{w \in \partial \Omega} |w - w_0|.
$$
As a consequence of the Distortion Theorem, we have that if $g:\D\to G$ is
a conformal map, then
\begin{align} \label{delta}
(1-|z|^2)|g'(z)|/4\le \delta_G(g(z))\le (1-|z|^2)|g'(z)|;
\end{align}
see \cite [Cor 1.4]{Pom}.

The following simple consequence will also be useful.

\begin{lem} \label{distortion2}
Let $f: \Omega_1 \to \Omega_2$ be a conformal map between  simply connected domains
and let $a \in \Omega_1$. Then
\begin{align*}
|f'(a)| \approx \delta_{\Omega_2} (f(a))/\delta_{\Omega_1} (a).
\end{align*}
\end{lem}
\begin{proof}
Let $h_1:\D\to \Omega_1$ be a Riemann map with $h_1(0)=a$ and let $h_2=f\circ h_1$.
Then $|f'(a)|=|h_2'(0)|/|h_1'(0)|\approx \delta_{\Omega_2} (f(a))/\delta_{\Omega_1} (a)$,
by (\ref{delta}).
\end{proof}

Next, we state a lemma from \cite{GP} that provides a geometric estimate
of the hyperbolic distance on a simply connected domain.
\begin{lem}\label{DistLem}\cite[Lemma 2.1]{GP}; see also \cite[\S 9.5]{Sh}.
Let $G$ be a simply connected proper subset of $\C$ and let $w_1, w_2\in G$.  Then
$$
\rho_G(w_1,w_2)\ge\frac 12\log\left(1+\frac{|w_1-w_2|}{\min\{\delta_G(w_1),\delta_G(w_2)\}}\right).
$$
\end{lem}

Denote by $\ell(\gamma)$ the arclength of a rectifiable curve $\gamma\subset\C$, and for $s>0$
define
$$
\gamma[s] = \{z\in\C\,:\, \delta_{\widetilde\gamma}(z)<s\},
$$
where $\widetilde\gamma$ denotes the complement of $\gamma$.
We let
$A(E)$ denote the
area of a measurable set $E$, and use $dA$ for integration with respect to Lebesgue area measure.

\begin{lem} \label{estDeltaInt}
Let $\gamma$ be a rectifiable curve,  $J$ be a line segment, let $s>0$, and let $q>-1$.
Then the following estimates hold:
$$
\int_{\gamma[s]}\delta_{\widetilde\gamma}^q(z)\, dA(z)\lesssim s^{q+1}(\ell(\gamma)+s);
$$
and,
$$
\int_{J[s]}\delta_{\widetilde J}^q(z)\, dA(z)\approx s^{q+1}(\ell(J)+s).
$$
\end{lem}
\begin{proof}
First, since $\gamma[t]$ can be covered by at most $1+\ell(\gamma)/t$ disks of
radius $2t$, we see that
$$
A(\gamma[t])\lesssim t(\ell(\gamma)+t), \quad t>0.
$$
  Hence
\begin{align*}
\int_{\gamma[s]}\delta_{\widetilde\gamma}^q(z)\, dA(z)
&= \sum_{n=0}^\infty
 \int_{\gamma[2^{-n}s]\setminus \gamma[2^{-n-1}s] }\delta_{\widetilde\gamma}^q(z)\, dA(z) \\
  &\lesssim \sum_{n=0}^\infty (2^{-n}s)^q \,2^{-n}s(\ell(\gamma)+2^{-n}s),
\end{align*}
and summing the geometric series completes the proof of the first estimate.  The proof of
the second is the same, except that now $A(J[t])\approx t(\ell(J)+t)$.
\end{proof}

We will need two results from univalent function theory, which we state here as lemmas.
The first is a result of Gehring and Hayman which says that hyperbolic geodesics are not much
longer than euclidean geodesics.

\begin{lem}\cite[Theorem 4.20]{Pom}  \label{GehringHayman}
Let $C$ be a curve in a simply connected domain $G \subsetneqq \C$ and $\gamma$ the noneuclidean segment in $G$ with the same endpoints as $C$.  There exists a universal constant $K$ such that  $\ell(\gamma) \leq K \ell(C)$.
\end{lem}

The second lemma, stating that euclidean disks are hyperbolically convex, is related to
results of J\o rgensen and Pommerenke; see \cite[Exercise 4.6.1]{Pom}.

\begin{lem}  \label{Jorgensen}
Let $B$ be an open euclidean disk in a domain $G \subsetneqq \C$.  If $\gamma$ is a noneuclidean segment of $G$ with endpoints in $B$, then $\gamma \subset B$.
\end{lem}

We now use these lemmas to prove the next result, which says that
the pullback of a noneuclidean segment contained in a euclidean subdisk of a domain
satisfies a chord-arc inequality.

\begin{lem} \label{chordarc}
Let $\Omega \subsetneqq \C$ be a simply connected domain with Riemann map $h: \D \to \Omega$.  Let $B$ be a euclidean disk in $\Omega$.  Suppose $\eta$ is a noneuclidean segment of $B$ with endpoints $w_1$ and $w_2$.  There exists a universal constant $K_1$ such that
  $$\ell(h^{-1}(\eta)) \leq K_1 |h^{-1}(w_2) - h^{-1}(w_1)|.$$
\end{lem}
\begin{proof}
Define $z_1 = h^{-1}(w_1)$ and $z_2 = h^{-1}(w_2)$, points in $\D$.  The noneuclidean segment of $\D$ that connects $z_1$ and $z_2$ is a radial segment or arc of a circle. Its image $\gamma$ under $h$ is a noneuclidean segment of $\Omega$, and
  $$\ell(h^{-1}(\gamma)) \leq \frac{\pi}{2} |z_1-z_2|.$$
Note that $\eta$ and $\gamma$ have the same endpoints, but $\eta$ is a noneuclidean segment of $B$ and $\gamma$ is a noneuclidean segment of the larger domain $\Omega$. Denote the restriction of $h^{-1}$ to $B$
  $$\phi = h^{-1} |_{B}.$$
The curve $\phi(\eta)$ is  a noneuclidean segment of the domain $\phi(B)$ with endpoints $z_1$ and $z_2$.  By Lemma \ref{Jorgensen}, $\gamma$ is also contained in $B$, so $\phi(\gamma)$ is contained in $\phi(B)$. By Lemma \ref{GehringHayman}
  $$\ell(\phi(\eta)) \leq K \ell(\phi(\gamma)).$$  Thus
  $$\ell(\phi(\eta)) \leq \frac{\pi}{2} K |z_1-z_2|.$$
\end{proof}

Next, we formulate a version of \cite[Cor. 1.5]{Pom} that will be
convenient for later application.

\begin{lem} \label{distortion} There exists $C>1$ such that, for any 
simply connected domain $G\subsetneqq \C$, any Riemann map
$h: \D \to G$, and any $w \in G$,
\begin{align*}
B(w,\delta_G(w)/C) \subset \Delta_G(w,1) \subset B(w,C\delta_G(w)).
\end{align*}
  Also, for all $z_0 \in \D$,
\begin{align*}
|h'(z)| \approx |h'(z_0)|, \qquad z \in \Delta_\D(z_0,1),
\end{align*}
with constants of comparison independent of $z_0$.
\end{lem}

\medskip

\subsection{The Besov spaces $\B_p$}
 For each $p > 1$, the Besov space $\B_p$ is the Banach space in $H(\D)$ with semi-norm
  $$\| f \|_p =  \left(\int_{\D} |f'(z)|^p (1 - |z|^2)^{p-2} \, dA(z)\right)^{1/p}.$$
The Dirichlet space is $\B_2$, and $p < q$ implies $\B_p \subset \B_q$ (see Ex. 5.36 in \cite{Zh}), which, with the closed graph theorem,  implies
  \begin{align} \label{pqnorms}
  \| f \|_q \lesssim \| f \|_p.
  \end{align}

In our proofs we will also need to use Carleson measures for Besov spaces.
For an arc $I \subseteq \partial \D$, define the \textit{Carleson square} associated with $I$ to be
  $$S(I) = \{ re^{i\theta}: 1-\ell(I) < r < 1, e^{i\theta} \in I \}.$$
A positive measure $\mu$ on $\D$ is called a \textit{$\B_p$-Carleson measure} if
 \begin{align*}
 \sup_{f \in \B_p} \frac{\|f'\|_{L^p(\mu)} }{\|f\|_{p}}= A_\mu  <\infty.
 \end{align*}

In \cite[Theorem 13]{AFP}, it is shown that $\mu$ is a $\B_p$-Carleson measure if and only if
\begin{align}\label{CarlConst}
\sup_{I\subset \partial \D} \frac{\mu(S(I))}{\ell(I)^p}= B_\mu  <\infty,
\end{align}
and moreover,  $A_\mu^p \approx B_\mu$.  The constants $A_\mu$ and $B_\mu$ are called Carleson
constants for $\mu$.

 Applying (\ref{delta}), a composition operator $C_\phi$
induced by a univalent selfmap $\phi$ of $\D$
is bounded on $\B_p$ if there exists a constant $C$ such that
\begin{align}
   \begin{split} \label{normest}
\|C_\phi(f)\|_{p}^p &= \int_{\D} |f'(\phi(z))|^p \, |\phi'(z)|^{p-2} (1-|z|^2)^{p-2} |\phi'(z)|^2 \, dA(z)\\
&\approx \int_{\phi(\D)} |f'(w)|^p \delta_{\phi(\D)}^{p-2}(w) \, dA(w)\\
&\leq C \|f\|_{p}^p,
   \end{split}
\end{align}
where $w = \phi(z)$.
 Hence, $C_\phi$ is bounded if and only if $\mu$ given by
\begin{align}  \label{defn mu}
  d\mu(w) = \delta_{\phi(\D)}^{p-2}(w)\chi_{\phi(\D)}(w)dA(w)
\end{align}  
is $\B_p$-Carleson, where $\chi_{\phi(\D)}$ is the characteristic function of $\phi(\D) \subset \C$,
and in this case
\begin{align} \label{Ccarleson}
\|C_\phi\|_p^p \approx A_\mu^p \approx B_\mu.
\end{align}

From the computation in (\ref{normest}), with $f(z)=z$, we see that
\begin{align} \label{Bpnorm}
 \|\phi\|^p_p = \|C_{\phi}(z)\|^p_p \approx \int_{\phi(\D)} \delta_{\phi(\D)}^{p-2}(w) dA(w),
\end{align}
whence
\begin{align} \label{Bpcarleson2}
\text{
$\phi \in \B_p$ implies $\mu(S(I)) = o(1), \quad \ell(I) \to 0$.}
\end{align}
The difference between (\ref{Bpcarleson2}) and (\ref{CarlConst}) motivated our work leading to Theorem \ref{compgp}.

Note that for $p \ge 2$ and a univalent selfmap  $\phi$ of $\D$, use of the Schwarz-Pick lemma gives
\begin{align}
	\begin{split} \label{contraction}
	\| C_{\phi} f \|^p_p &=  \int_{\D}  |f'(\phi(z))|^p (1-|z|^2)^{p-2} |\phi'(z)|^{p-2} |\phi'(z)|^2 \, dA(z)\\
 	&\leq \int_{\D}  |f'(\phi(z))|^p (1-|\phi(z)|^2)^{p-2} |\phi'(z)|^2 \, dA(z)\\
	&= \int_{\phi(\D)} |f'(w)|^p (1-|w|^2)^{p-2} \, dA(w)\\
	&\leq \|f\|_p^p.
	\end{split}
\end{align}

\section{Composition semigroups on $\B_p$}

In \cite{Sis2}, Siskakis shows that every semigroup of functions $\{\phi_t\}$ induces a strongly continuous semigroup of operators on $\B_2$.  The main idea of the proof is to show that
$\lim_{t\to0}  \|\phi_t(z)-z\|_2=0$ and that this implies that $\{C_t\}$ is strongly continuous
on all polynomials.
The next result provides a partial extension  to all $p>1$.

\begin{prop} \label{geq1}
 Let $p>1$ and denote by $\mathcal{P}$ the vector space
of polynomials and let  $\{\phi_t\}$ be any semigroup of self-maps of $\D$.  The following are equivalent:
\begin{enumerate}
\item [(1)] $\mathcal{P} \subset (\phi_t,\B_p)$;
\item [(2)] $\lim_{t\to0}  \|\phi_t(z)-z\|_p=0$.
\end{enumerate}
\end{prop}

\begin{proof}
That (1) implies (2) is of course immediate, since $p(z)=z$ is in $\mathcal{P}$.  For the converse,
by linearity it suffices to show that (2) implies that all monomials belong to $(\phi_t,\B_p)$. For $k$ a positive integer,
\begin{align*}
	 \| \phi_t^k(z) - z^k \|^p_p  &=  k^p \int_{\D} | \phi_t^{k-1}(z) \phi'_t(z) -  z^{k-1} |^p
             (1-|z|^2)^{p-2} \, dA(z)\\
	&\leq 2^{p-1} k^p  \int_{\D} | \phi_t^{k-1}(z) \phi'_t(z) -  \phi_t^{k-1}(z)  |^p (1-|z|^2)^{p-2} \, dA(z)\\
	& \quad + 2^{p-1} k^p \int_{\D} | \phi_t^{k-1}(z)  -  z^{k-1} |^p (1-|z|^2)^{p-2} \, dA(z) .
\end{align*}
Since $|\phi_t(z)|\le 1$, the first term is bounded by $2^p k^p\|\phi_t(z)-z\|_p$, which by (2) has
a limit of 0 as $t\to 0$.  The second term also has a limit of 0, since it is known that
\begin{align} \label{unif}
\lim_{t\to0} \|\phi_t(z) - z\|_{H^\infty}=0,
\end{align}
for every semigroup $\{\phi_t\}$; see  \cite[Prop. 3.2]{Gum} or \cite[Thm. 1.2]{AJS2}.
Hence every monomial $z^k$ belongs to $(\phi_t,\B_p)$, and the proof is complete.
\end{proof}

It is known that $\mathcal{P}$ is dense in $\B_p$; see \cite[Prop. 2]{AFP}.  Hence the following
is an immediate corollary of Proposition \ref{subspSCC} and Proposition \ref{geq1}.

\begin{cor} \label{p>1}
Let $p>1$ and suppose that the semigroup $\{\phi_t\}$ satisfies
\newline $\sup_{0\le t\le 1} \|C_t\|_{\B_p}<\infty$.
Then $[\phi_t,\B_p]=\B_p$ if and only if $\lim_{t\to0}  \|\phi_t(z)-z\|_p=0$.
\end{cor}

\begin{cor} \label{geq2}
For $p \geq 2$ and every semigroup $\{\phi_t\}$, $[\phi_t,\B_p]=\B_p$.
\end{cor}

\begin{proof}
As mentioned above, the case $p = 2$ was done in \cite{Sis2} by showing
that  $\lim_{t\to0}  \|\phi_t(z)-z\|_2=0$.  From (\ref{pqnorms}), it follows
that $\lim_{t\to0}  \|\phi_t(z)-z\|_p=0$ for $p \geq 2$. Also, from (\ref{contraction})
we have that $\|C_t\|_{\B_p}\le 1$ for $p\ge2$.  Hence the result is a consequence of
Corollary \ref{p>1}.
\end{proof}

We next give an example that will be used to construct a semigroup $\{\phi_t\}$
such that $\{C_t\}$ does not act on any $\B_p$, $1<p<2$.
%In order to do that we will need the following Lemmas.

\begin{prop} \label{comb}
There exists a domain $\Omega$, starlike with respect to 0, such that the Riemann map $h: \D \to \Omega$ with $h(0) = 0$ and $h'(0)>0$ is not in any $\B_p$, $1 < p < 2.$
\end{prop}

\begin{proof}
Let $\{x_n\}_{n\ge1}$ be a decreasing sequence of positive numbers such that $\sum_n x_n = 1$ and $\sum_n x_n^{p-1} = \infty$ for all $p < 2$.  For example, we can let $x_1=1/2$ and $x_n=Cn^{-1}\log^{-2}n$, $n\ge 2$, where
$C^{-1}=2\sum_{n\ge 2} n^{-1}\log^{-2}n$.

Define $\theta_n =  \sum_{k=n}^{\infty} x_k $, and let $\Omega$ be the domain
$$
  \Omega = \D \setminus \bigcup_{n=1}^{\infty} J_n,
$$
where $J_n = \left[ e^{i\theta_n}/2,e^{i\theta_n}\right)$. Notice that
$\delta_\Omega(w) \lesssim \delta_{\widetilde{J_n}}(w)$ for $ w \in J_n[x_n/2]\cap\Omega$.

   Let $h : \D \to \Omega$ be the Riemann map with $h(0) = 0$, $h'(0) > 0$.  Beginning with (\ref{Bpnorm}) and employing Lemma \ref{estDeltaInt},
\begin{align*}
	\| h \|^p_p & \approx \int_\Omega \delta_\Omega^{p-2}(w) \, dA(w)
	\geq \sum_n \int_{J_n[x_n/2]} \delta_{\widetilde{J_n}}^{p-2}(w) \, dA(w)
	\approx \sum_n x_n^{p-1},
\end{align*}
implying $h \notin \B_p$, $1<p<2$.
\end{proof}

%\begin{rem} \label{finiteslits}
%Note that if the number of slits is finite in the above example, then $h: \D \to \D$ is in $\B_p$ and $C_h$ %is bounded on $\B_p$.
%\end{rem}

We now analyze a semigroup induced by $h$ from Proposition \ref{comb} to obtain a counterexample to the $1< p < 2$ version of Corollary \ref{geq2}, in the sense that not every semigroup of functions $\{\phi_t\}$ induces a semigroup of bounded composition operators $\{C_t\}$.
\begin{prop} \label{compgpProp}
There exists a semigroup $\{\phi_t\}$ that  does not act on any $\B_p$, $1<p<2$.
\end{prop}
\begin{proof}
With  $h$ the univalent map in Proposition \ref{comb}, consider the functions
  $$\phi_t(z) = h^{-1}(e^{-t}h(z)), \quad t\ge 0.$$
Since $\Omega=h(\D)$ is starlike, we have that $\{\phi_t\}\subset H(\D)$ and $\phi_t(\D)\subset \D$.
  It is easy to verify that (SG1), (SG2), and (SG3) hold, and so $\{\phi_t\}$ is a semigroup
of analytic functions on the disk.
(In fact, every semigroup of analytic functions on $\D$ can be represented in a similar way;
see \cite{BCDM} or \cite{Sis} for more on this.)
  Note that $|\phi_t(z)|\to 0$ uniformly on $\D$ as $t\to\infty$,
since $h(\D)$ is bounded.  Also,
\begin{align} \label{phi=h}
	\phi'_t(z) = \frac{e^{-t}h'(z)}{h'(\phi_t(z))}.
\end{align}

For $z \in \D$, $h'(z) = c_1 + O(z)$ as $z \to 0$, where $c_1 > 0$, and so
 there exists $t_0$ such that if $t\ge t_0$, then $|h'(\phi_t(z))| \lesssim 2c_1$, $z\in\D$.
 Thus, from (\ref{phi=h}),
\begin{align*}
  	|\phi'_t(z)| \ge \frac{e^{-t} |h'(z)|}{2c_1}, \quad t\in \D,
\end{align*}
for $t\ge t_0$.  Thus,
\begin{align*}
\|C_t \,z\|_p^p	=\|\phi_t\|_p^p = \int_{\D} |\phi'_t(z)|^p(1-|z|^2)^{p-2} \, dA(z)
	\ge \left( \frac{e^{-t}}{2c_1}\right)^p \|h\|^p_p = \infty, \quad t\ge t_0,
\end{align*}
and hence
$C_t$ does not act on $\B_p$ when $t\ge t_0$.  Since $C_t = C_{t/n}^n$, $n\in \mathbb N$,
it follows that $C_t$ does not act on $\B_p$ for any $t>0$.
\end{proof}

%The intuition behind Proposition \ref{compgpProp} is that for large $t$ the image of $\phi_t$ is nearly a small constant times the image of $h$.  The idea applies to the disk algebra $A$ as well as $\B_p$.

In the previous example the semigroup $\{\phi_t\}$ was not in $\B_p$. In the next example we show that even when
$\{\phi_t\}\subset \B_p$ the operators $\{C_t\}$ may not be bounded.

%%%%%%%%%

\begin{thm} \label{compgp}
Let $1<p<2$.
There exists a semigroup $\{\phi_t\}\subset \B_p$ such that $\{\phi_t\}$ does not act on $\B_p$.
%$1<p<2$.
\end{thm}

\begin{proof}
%%%%%%%%%%%%%%%%%%%%%%%%%%%%%%%%%%%%%%%%%%%%%%%%%%%
Let $\theta_0 = 0$ and for integers $n\neq 0$, let
$$
\theta_n=\sum_{k=n}^\infty k^{-4/(p-1)},\quad n\ge1,\quad\text{and},\quad
  \theta_{n}=-\theta_{-n}, \quad n \le -1.
$$
Since $1<p<2$, we have that $\theta_1< \pi/2$.  Also, note that

\begin{align} \label{choice}
\theta_n\approx n^{-(5-p)/(p-1)}\quad\text{and}\quad
\Delta \theta_n =\theta_n-\theta_{n+1}= n^{-4/(p-1)},\quad n\ge 1.
\end{align}

Define the infinite slits $\{S_n\}$ by
  $$S_n = \{re^{i\theta_n} : r \geq 1\},$$
and let $\Omega$ be the starlike domain $\C \setminus \cup_{n \in \mathbb Z} S_n $.
%Let $\Omega$ be the starlike domain $\C - \cup_n (S_n \cup \overline{S_n})$.
Define $h: \D \to \Omega$ to be the Riemann map such that $h(0) = 0$ and $h'(0)>0$,
so that $h(r) > 0$ for $0 < r < 1$ by symmetry.

 The selfmap  $w \mapsto e^{-t}w$ of $\Omega$ creates radial extensions
 $Q_n = \{ re^{i\theta_n}: e^{-t} \leq r \leq 1 \}$ of $S_n$ that are mapped by $h^{-1}$ to continua $R_n = h^{-1}(Q_n)$.
 See Figure \ref{figure1}.

In addition, we define
$L = h^{-1}(\{r\,:\, e^{-t}\le r\le 1\})$.  From Lemma \ref{chordarc},
there is a constant $K$ such that
\begin{align}\label{bddLength}
\ell(R_n)\le K,\quad 1\le n<\infty, \quad \text{and}\quad \ell(L)\le K.
\end{align}
In fact, we can take $K=2K_1$, where $K_1$ is the universal constant from  Lemma \ref{chordarc}.

Let $\phi_t (z) = h^{-1}(e^{-t} h(z)), t \geq 0$.
Again, similarly to the related example in Proposition \ref{compgpProp}, it is easy to verify that
$\{\phi_t\}$ is a semigroup of analytic functions on $\D$.

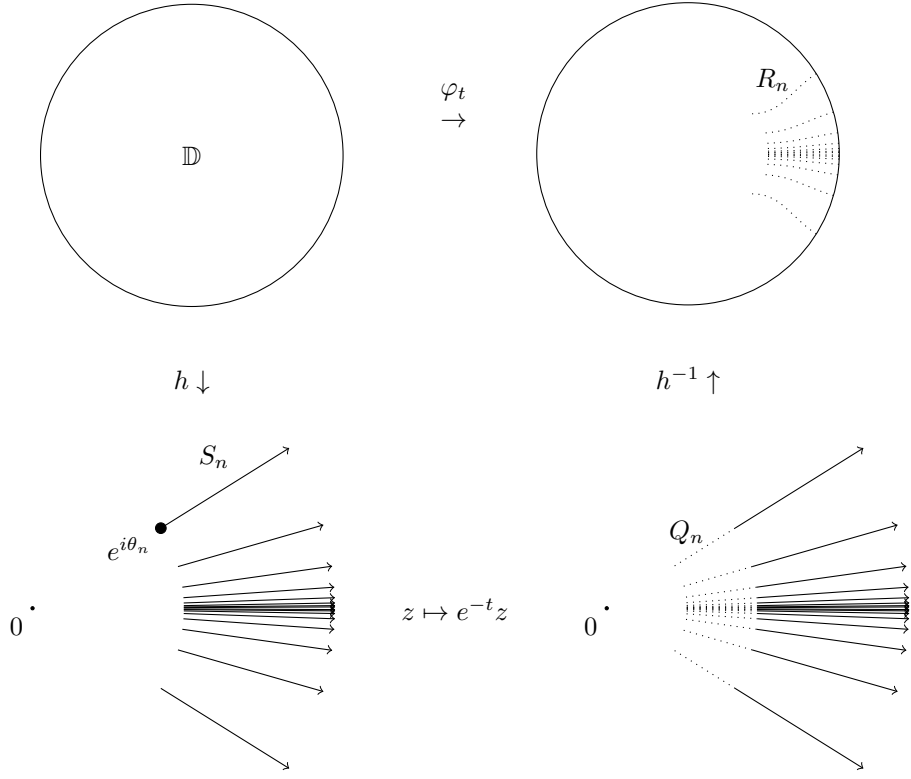
\begin{figure}[hbt!]
\begin{subfigure}{0.4\textwidth}
\begin{tikzpicture}[scale=2]
\draw (0,0) circle (1cm);
\draw (0,0) node{$\D$};
\draw (0,-1.5) node{$h \downarrow$};
\end{tikzpicture}
\end{subfigure}
\begin{subfigure}{0.1\textwidth}
\begin{tikzpicture}[scale=2]
\draw (0,2) node{$\phi_t$};
\draw (0,1.8) node{$\rightarrow$};
\draw(0,0) node{$$};
\end{tikzpicture}
\end{subfigure}
\begin{subfigure}{0.4\textwidth}
\begin{tikzpicture}[scale=2]
\draw (0,0) circle(1cm);
\foreach \n in {-1,...,5}
{
\draw[dotted] (2^\n: 1) to[out=2^\n-180,in=0] (2^\n: 0.5);
\draw[dotted] (-2^\n: 1) to[out=-2^\n+180,in=0] (-2^\n: 0.5);
}
\draw(0.56,0.49) node{$R_n$};
%\filldraw (32:0.5) circle (1pt) node[anchor=north east]{$z_n$};
%\filldraw (32:1) circle (1pt) node[anchor=south west]{$e^{id_n}$};
\draw (0,-1.5) node{$h^{-1} \uparrow$};

\end{tikzpicture}

\end{subfigure}
\vskip\baselineskip
\begin{subfigure}{0.4\textwidth}
\begin{tikzpicture}[scale=2]
\foreach \n in {-2,...,5}
{
\draw[->](2^\n: 1) -- (2^\n: 2);
\draw[->](-2^\n: 1) -- (-2^\n: 2);
}
\draw (1.2,1) node{$S_n$};
\filldraw (32: 1) circle (1pt) node[anchor=north east]{$e^{i\theta_n}$};
\filldraw (0,0) circle(0.3pt) node[anchor=north east]{$0$};
\end{tikzpicture}
\end{subfigure}
\begin{subfigure}{0.1\textwidth}
\begin{tikzpicture}[scale=2]
\draw (0,1) node{$z \mapsto e^{-t}z$};
\draw(0,0) node{$$};
\end{tikzpicture}
\end{subfigure}
\hfill
\begin{subfigure}{0.4\textwidth}

\begin{tikzpicture}[scale=2]
\foreach \n in {-2,...,5}
{
\draw[->](2^\n: 1) -- (2^\n: 2);
\draw[->](-2^\n: 1) -- (-2^\n: 2);
\draw[dotted] (2^\n: 1) -- (2^\n: 0.5);
\draw[dotted] (-2^\n: 1) -- (-2^\n: 0.5);
}
\draw (0.71,0.64) node[anchor = north east]{$Q_n$};
\filldraw (0,0) circle(0.3pt) node[anchor=north east]{$0$};
\end{tikzpicture}
\end{subfigure}
\caption{$\phi_t(z) = h^{-1}(e^{-t}h(z))$}  \label{figure1}
\end{figure}

\medskip

First we will show that  $\phi_t \in \B_p$.
We introduce the  notation
\begin{align*}
U&=  h^{-1}(e^{-t}\D);\\
V&=\{ z\in \phi_t(\D) \,:\, |h(z)|>e^{-t}\,\,\text{and}\,\,
       \delta_{\phi_t(\D)}(z) = \delta_{\widetilde{ L}}(z)\};\\
G&=\{ z\in \phi_t(\D) \,:\,  |h(z)|>e^{-t}\,\,\text{and}\,\,
      \delta_{\phi_t(\D)}(z) = \delta_{\D}(z)\};\\
H_k &=\{ z\in \phi_t(\D) \,:\, |h(z)|>e^{-t}\,\,\text{and}\,\, \delta_{\phi_t(\D)}(z) =
   \delta_{\widetilde{ R_k}}(z)\},  k \in \mathbb Z \setminus \{0\}.
\end{align*}

It is easily checked that
$\overline U= \{ z\in \phi_t(\D) \,:\,  |h(z)|\le e^{-t}\}$,
and so
$$
\phi_t(\D) =\overline U \cup V\cup G\cup \bigcup_{k \neq 0} H_k.
$$
Hence,  to show $\phi_t \in \B_p$, by (\ref{Bpnorm}) it suffices to show that the integral of $\delta_{\phi_t(\D)}^{p-2}$
over each these sets is finite.

From the Schwarz lemma applied to the restriction of $h^{-1}$ to $\D$, we have that
$U\subset e^{-t}\D$.  Hence $h$ is a bi-Lipschitz map from $U$ to $e^{-t}\D$, and
$$
I_1 = \int_{U}\delta_{\phi_t(\D)}^{p-2}\, dA
    \le \int_{U}\delta_{U}^{p-2}\, dA
      \approx \int_{e^{-t}\D}\delta_{e^{-t}\D}^{p-2}\, dA<\infty.
$$

Next, $V\subset L[2]$, and using Lemma \ref{estDeltaInt} and (\ref{bddLength}) we get that
$$
I_2= \int_{V}\delta_{\phi_t(\D)}^{p-2}\, dA
     \le \int_{L[2]}\delta_{\widetilde L}^{p-2}\, dA
\lesssim 2^{p-1}(\ell(L)+2)<\infty.
$$

Regarding the third integral, since $\delta_{\phi_t(\D)}=1-|z|$ on $G$ and $p>1$,
$$
I_3=\int_{G} \delta_{\phi_t(\D)}^{p-2} \, dA=\int_{G} (1-|z|)^{p-2} \, dA(z)
\le
\int_{\D} (1-|z|)^{p-2} \, dA(z)<\infty.
$$

Finally, we turn to the integral over $\bigcup_{k \neq 0} H_k$. We will need to use some facts about harmonic measure.
  %Consider $a \in H_k$, where $2\le k<\infty$,
 %and for simplicity write
%$$
%\Delta = \Delta_{\phi_t(\D)}(a, 1).
%$$

 For a domain $\mathcal D$ with measurable set $E \subset \partial \mathcal D$, we denote by $$\omega(w_0,E,\mathcal D)= \omega_\mathcal D(w_0,E)$$ the harmonic measure of $E$ with respect to $w_0 \in \mathcal D$.  See \cite{GM} for more information.

Let $z \in H_k$, where $2\le k<\infty$,
 and for simplicity write
$$
\Delta = \Delta_{\phi_t(\D)}(z, 1).
$$
With $z \in H_k$, $k\ge2$, we have that $e^{-t} < |h(z)|$ and
we first consider the case that
$|h(z)|\le 1$ as well.
Then $h(\Delta)= \Delta_{e^{-t}\Omega}(h(z),1)$ and using Lemma \ref{distortion} we see that
$$
\diam (h(\Delta)) \lesssim \Delta \theta_{k}.
$$
It follows that
$$
\omega_{\D\setminus \Delta}(0,\Delta) = \omega_{\Omega \setminus h(\Delta)}(0,h(\Delta))
\lesssim \left( \diam (h(\Delta)) \right)^{1/2}
\lesssim \Delta \theta_{k}^{1/2},\quad |h(z)|\le 1,
$$
where \cite[Cor. III.9.3]{GM} was used for the first inequality.
Also, by the radial  Hall's Lemma (see \cite[p.125]{GM}) and Lemma \ref{distortion},
$$
\omega_{\D \setminus \Delta}(0,\Delta) \gtrsim m(\Delta^*)
\approx \delta_{\phi_t(\mathbb D)}(z),
$$
where $\Delta^*$ is the radial projection of $\Delta$ to $\partial \D$, and $m$ is Lebesgue measure on $\partial \D$.  Hence
\begin{align} \label{case h(z) small}
\delta_{\phi_t(\mathbb D)}(z)\lesssim \Delta \theta_{k}^{1/2}, \quad z\in H_k  \quad k\ge 2,
\end{align}
when $e^{-t} < |h(z)|\le 1$.

Now consider $z \in H_k$, $k\ge 2$,  with $|h(z)|> 1$,
so that $\delta_\Omega(h(z))\le |h(z)|\Delta\theta_k/2$.  From Lemma \ref{DistLem} we then get that
$$
\rho_\Omega(0,h(z))\ge \frac12\log\left(1+\frac {|h(z)|}{|h(z)|\Delta\theta_k/2}\right)
    \ge \frac12\log\left(\frac1{\Delta\theta_k}\right).
$$
Thus, since $\rho_\D(0,z)=\log((1+|z|)/(1-|z|))$,
$$
1-|z| \approx \exp(-\rho_\D(0,z))
    = \exp(-\rho_\Omega(0,h(z)))
      \le  \Delta \theta_{k}^{1/2}.
$$
Hence, for $z \in H_k$ and $|h(z)|>1$,
$$
\delta_{\widetilde{ R_k}}(z)=\delta_{\phi_t(\D)} (z)\le 1-|z|\lesssim \Delta \theta_{k}^{1/2},
   \quad k\ge 2.
$$

Combined with (\ref{case h(z) small}), this shows that for $z \in H_k$, $2\le k<\infty$, $\delta_{\phi_t(\D)}(z) \lesssim \Delta \theta_k^{1/2}= k^{-2/(p-1)}$.
It follows that there is a positive constant $C$ such that
$H_k\subset R_k[Ck^{-2/(p-1)}]$, $2\le k<\infty$.  Also, $H_1\subset R_1[2]$. Hence,
again using Lemma \ref{estDeltaInt} and (\ref{bddLength}),
\begin{align*}
I_4&= \int_{\bigcup_{k \neq 0} H_k}\delta_{\phi_t(\D)}^{p-2}\, dA
  \le 2 \int_{H_1}\delta_{\widetilde{ R_1}}^{p-2}dA\,
     + 2 \sum_{k=2}^\infty\int_{H_k}\delta_{\widetilde{ R_k}}^{p-2}\, dA\\
  &\lesssim \int_{ R_1[2]}\delta_{\widetilde{ R_1}}^{p-2}dA
    + \sum_{k=2}^\infty \int_{R_k[Ck^{-2/(p-1)}]}\delta_{\widetilde{ R_k}}^{p-2}\, dA\\
  &\lesssim \ell(R_1)+\sup_{k\ge2}\ell(R_k)\sum_{k=2}^\infty k^{-2}<\infty.
\end{align*}

 Hence
$$
 \int_{\phi_t(\D)} \delta_{\phi_t(\D)}^{p-2} \, dA =I_1+I_2+I_3+I_4 < \infty,
$$
 and this completes the proof that $\phi_t \in \B_p.$

%%%%%%%%%%%%%%%%%%%%%%%%%%%%%%%%%%%%%%%%
\smallskip

Next, we show that the operators $\{C_t\}$ are not bounded on $\B_p$ by showing that
(\ref{CarlConst}) fails.
Let $I\subset \partial \D$ be an arc centered at $1$, and let $\mathcal{S}(I)$  be the corresponding Carleson square.
By Lemma \ref{distortion} we can choose $z_I\in (0,1)$
%\begin{align}
  such that $1-z_I\approx \ell(I)$ and $\Delta_{\D}(z_I,1) \subset \mathcal{S}(I)$.  Let  $\Delta_I=\Delta_{\D} (z_I,1)$.
%Here $\Delta(z_0,1)$ is the hyperbolic disc in $\D$, centered at $z_0$ with radius $1$.
%z_0 = 1 - \frac{\ell(I)}{2}, \text{ and } w_0 = h(z_0).
%\end{align}

%The function $h$ behaves roughly like a dilation by factor $|h'(z_0)|$ in $\Delta$.  We use a change of variables $w = h(z)$, with $w_0 = h(z_0)$, and apply Lemmas \ref{distortion} and \ref{distortion2}.

Then, with $\mu$ the measure given in \eqref{defn mu},
\begin{align*}
\mu (\mathcal{S}(I) ) & \ge \mu(\Delta_I)
= \int_{\Delta_I} \delta_{\phi_t(\D)}^{p-2}(z) \chi_{\phi_t(\D)}(z)\, dA(z)\\
&\approx |h'(z_I)|^{-2} \int_{\Delta_I} \delta_{\phi_t(\D)}^{p-2}(z) \chi_{\phi_t(\D)}(z) |h'(z)|^2 \, dA(z),
\end{align*}
where Lemma \ref{distortion} was used for the last estimate.  Hence, using the
change of variable $w = h(z)$ with $w_I = h(z_I)$, and Lemma \ref{distortion2}, we get
that
\begin{align}\label{lowerbd}
\mu (\mathcal{S}(I) )\gtrsim |h'(z_I)|^{-p} \int_{\Delta_{\Omega}(w_I,1) } \delta_{e^{-t}\Omega}^{p-2}(w)  \chi_{e^{-t}\Omega}(w) \, dA(w).
\end{align}
Also note that $w_I\in(0,1)$.
  For each $n\ge 1$, introduce the regions
$$\mathcal{W}_n = \{w : 0 < \arg w < \theta_n, w_I < |w| < w_I + \theta_n\},$$
where the integral above is easy to estimate, and introduce the notation
$$n_I = \inf \{n : \mathcal{W}_n \subset \Delta_{\Omega}(w_I,1) \}
\quad \text{and}\quad \theta_I=\theta_{n_I}.$$
Using (\ref{delta}) and Lemma \ref{distortion},  we see that
\begin{align}\label{est theta}
|h'(z_I)| \, \ell(I) \approx |h'(z_I)|(1-z_I)\approx \delta_\Omega(w_I)\approx \theta_I.
\end{align}
Denote by $J_k$ the line segment $J_k=[w_I e^{i\theta_k}, ( w_I + \theta_I)e^{i\theta_k}].$
 From (\ref{lowerbd}), using that $\mathcal{W}_{n_I} \subset \Delta_{\Omega}(w_I,1)$ and
then Lemma \ref{estDeltaInt}, we get that
\begin{align*}
\mu (\mathcal{S}(I) )&\gtrsim
     \theta_I^{-p}\ell(I)^p\int_{\mathcal{W}_{n_I} } \delta_{e^{-t}\Omega}^{p-2}(w)  \chi_{e^{-t}\Omega}(w) \, dA(w)\\
   &\ge \theta_I^{-p}\ell(I)^p\sum_{k=n_I+1}^{\infty} \int_{J_k[\Delta \theta_k/2]}
   \delta_{\widetilde{J_k}}^{p-2}(w) \, dA(w)
  \approx \theta_I^{-p+1}\ell(I)^p\sum_{k=n_I}^{\infty} (\Delta \theta_k)^{p-1}.
\end{align*}

Using (\ref{choice}), this can be re-written
\begin{align}\label{CarlFails}
\mu (\mathcal{S}(I) )\gtrsim n_I^{2-p}\ell(I)^p.
\end{align}
From the Schwarz lemma applied to the restriction of $h^{-1}$ to $\D$, $z_I<w_I<1.$
Also, as $\ell(I)\to 0$, $z_I\to 1$ and so $w_I\to 1$ as well.  Hence, using (\ref{est theta}),
$\theta_I\approx\delta_\Omega(w_I) = 1-w_I\to0$ as $\ell(I)\to 0$.  Now (\ref{choice}) tells
us that $n_I\to\infty$ as $\ell(I)\to0$, and so (\ref{CarlFails}) means that (\ref{CarlConst})
fails, and hence $\{C_t\}$ does not act on $\B_p$.

\end{proof}

\begin{rem} \label{rmk}
For the semigroup $\{\phi_t\}$ in Theorem \ref{compgp}, we do not know if $[\phi_t,\B_p]$
exists.  There exist closed subspaces of strong continuity that are invariant under each $C_t$,
such as the one-dimensional subspace
of constant functions, but we do not know if there is a maximal one.  But
since $\{\phi_t\}$ does not act on $\B_p$,  if $[\phi_t,\B_p]$ exists it can not be $\B_p$.
\end{rem}

Before giving the other examples mentioned in the Introduction, we need a preliminary result.

\begin{prop} \label{fullMeasure}
Let $1<p<2$ and let
$\{\phi_t\} \subset \B_p$ be a semigroup such that $A(\phi_t(\D)) = A(\D)$ for sufficiently small $t > 0$.  Then $\| \phi_t(z) - z\|_p \to 0$ as $t \to 0$.
\end{prop}

\begin{proof}
Let $\eps > 0$ and for convenience introduce the notation
\begin{align*}
\delta_t = \delta_{\phi_t(\D)}.
\end{align*}
We assume there exists $t_0$ such that $\phi_t(\D)$ has full measure for $t \leq t_0$.  From (\ref{Bpnorm}),
\begin{align*}
\int_{\D} \delta_{t_0}^{p-2} \, dA = \int_{\phi_{t_0}(\D)} \delta_{t_0}^{p-2} \, dA
\approx \|\phi_{t_0}\|_p^p < \infty,
\end{align*}
which allows us to
pick $r \in (0,1)$ such that
\begin{align} \label{choice_of_r}
\int_{\D \setminus r\D} \delta_{t_0}^{p-2} \, dA < \eps.
\end{align}
Then
\begin{align*}
\| \phi_t(z) - z\|_p^p &= \int_{r\D} |\phi_t'(z) - 1|^p(1 - |z|^2)^{p-2} \, dA(z) \\
&+ \int_{\D \setminus r\D} |\phi_t'(z) - 1|^p(1 - |z|^2)^{p-2} \, dA(z)\\
&= I(t) + II(t).
\end{align*}
The term $I(t) \to 0$ as $t \to 0$, since $\phi_t' \to 1$ uniformly on compact sets by (\ref{unif})
and the Weierstrass convergence theorem. As for the other term,
\begin{align*}
II(t) &\lesssim \int_{\D \setminus r\D} |\phi_t'(z)|^p(1 - |z|^2)^{p-2} \, dA(z) + \int_{\D \setminus r\D}(1 - |z|^2)^{p-2} \, dA(z)\\
&= III(t) + IV.
\end{align*}
Using (\ref{delta}) and the change of variables $w = \phi_t(z)$, we have, for $t \leq t_0$,
\begin{align*}
III(t)
\approx \int_{\phi_t(\D \setminus r\D)}\delta_t^{p-2}(w) \, dA(w)
 \leq \int_{\D} \delta_{t_0}^{p-2}(w)\chi_{\phi_{t}(\D \setminus r\D)}(w) \, dA(w),
\end{align*}
because $t \leq t_0$ implies $\delta_t \geq \delta_{t_0}$.
The characteristic function $ \chi_{\phi_{t}(\D \setminus r\D)}(z) \to \chi_{\D \setminus r\D}(z)$ pointwise as $t \to 0$, and by (\ref{choice_of_r}) and dominated convergence we have
\begin{align*}
\limsup_{t \to 0} III(t) < \eps.
\end{align*}  Finally, using (\ref{choice_of_r}),
$IV < \eps$
 because $\delta_{t_0}(z) \leq 1-|z|$.  Since $\eps>0$ was arbitrary, this completes the proof.
\end{proof}

We now can give the example promised in the introduction of, for $1<p<2$, a semigroup
$\{\phi_t\}$ that does not act on $\B_p$ but for which
$(\phi_t,\B_p)$ is dense in $\B_p$.

%%%%%%%%%%
\begin{cor} \label{PnotB_p}
Let $1 < p < 2$. There exists a semigroup $\{\phi_t\} \subset \B_p$ such that
$(\phi_t,\B_p)$ is dense in $\B_p$, but $\{\phi_t\}$ does not act on $\B_p$.
\end{cor}
%%%%%%%%%%%%

\begin{proof}
Let $1 < p < 2$ and let $\{\phi_t\}$ be the semigroup constructed in
Theorem \ref{compgp}, so that $\{\phi_t\} \subset \B_p$,
$\{\phi_t\}$ does not act on $\B_p$, and $A(\phi_t(\D)) = A(\D)$ for all $t$.
Thus Proposition \ref{fullMeasure} tells us that
$\| \phi_t(z) - z\|_p \to 0$ as $t \to 0$, and so $\mathcal{P}\subset (\phi_t,\B_p)$
from Proposition \ref{geq1}. Since the polynomials are dense in $\B_p$, this completes the proof.
\end{proof}

We conclude the paper with another example promised in the introduction.

\begin{thm} \label{[]=B_p}
There exists a nontrivial semigroup $\{\psi_t\}$ such that $[\psi_t,\B_p]=\B_p$, $1<p<2$.
\end{thm}

\begin{proof}
Let
\begin{align*}
h(z) = \frac{4z}{(1+z)^2},
\end{align*}
so $h$ is a Riemann map of $\D$ onto $h(\D) = \C \setminus \{ r : r \ge 1\}.$ Let
\begin{align*}
\psi_t(z) = h^{-1}(e^{-t}h(z)),
\end{align*}
so that $\psi_t(\D) =\D\setminus[r_t,1)$, where $r_t = h^{-1}(e^{-t})$.
Just as in the proof of Proposition \ref{compgpProp}, we have that $\{\psi_t\}$ is
a semigroup of analytic functions on $\D$.
  It is easy
to verify, using Lemma \ref{estDeltaInt}, that the associated measures
$$d\mu_t(w) = \delta_{\psi_t(\D)}^{p-2}(w)\chi_{\psi_t(\D)}(w)dA(w)$$
are $\B_p$-Carleson measures with Carleson constants $B_{\mu_t}$ uniformly bounded for $t>0$.
Hence, from (\ref{Ccarleson}), the composition
semigroup $\{C_t\}$ induced by $\{\psi_t\}$ satisfies $\sup_{t>0}\|C_t\| <\infty$.
The result is now an immediate consequence of Proposition \ref{fullMeasure} and Corollary \ref{p>1}.
\end{proof}

%%%%%%%%%%%%%%%%%%%%%%%%%%%%%%%%%%%%%%%%%%%%%%%%%

 \end{document}